\documentclass[11pt,twoside,reqno]{amsart}
\allowdisplaybreaks
\usepackage{amsmath,amstext,amssymb,epsfig,multicol,enumerate}
\usepackage{graphicx}
\usepackage{mathrsfs}
\calclayout
\makeatletter
\renewcommand{\@seccntformat}[1]{\bf\csname the#1\endcsname.}
\renewcommand{\section}{\@startsection{section}{1}
	\z@{.7\linespacing\@plus\linespacing}{.5\linespacing}
	{\normalfont\upshape\bfseries\centering}}
\renewcommand{\@biblabel}[1]{\@ifnotempty{#1}{#1.}}
\makeatother
\theoremstyle{plain}
\newtheorem{thm}{Theorem}[section]
\newtheorem{lem}[thm]{Lemma}
\newtheorem{prop}[thm]{Proposition}

\theoremstyle{definition}

\newtheorem{defn}[thm]{Definition}

\usepackage{cancel}
\usepackage[parfill]{parskip}
\usepackage[german]{varioref}
\usepackage[all]{xy}
\usepackage{color}

\def \>{\succ}
\def \<{\prec}

\begin{document}	
\title[Imed Basdouri\textsuperscript{1}, Bouzid Mosbahi\textsuperscript{2}]{Matrix Representations of Derivations for Low-Dimensional Mock-Lie Algebras}
	\author{Imed Basdouri\textsuperscript{1}, Bouzid Mosbahi\textsuperscript{2}}
 \address{\textsuperscript{1}Department of Mathematics, Faculty of Sciences, University of Gafsa, Gafsa, Tunisia}
\address{\textsuperscript{2}Department of Mathematics, Faculty of Sciences, University of Sfax, Sfax, Tunisia}

 \email{\textsuperscript{1}basdourimed@yahoo.fr}
\email{\textsuperscript{2}mosbahi.bouzid.etud@fss.usf.tn}
	
	
	

	\date{\today}

\begin{abstract}
 In this work, we study the matrix representation of derivations for Mock-Lie algebras with dimensions up to four. Using matrix methods, we examine their structure and properties, showing how these derivations help us better understand the algebraic nature of Mock-Lie algebras.
\end{abstract}
\maketitle
\section{ Introduction}

A Mock-Lie algebra is a vector space equipped with a bilinear product that satisfies two key conditions: the product is commutative, and it satisfies the Jacobi identity. This Jacobi identity, which is also a fundamental property of Lie algebras, ensures a specific relationship involving three elements of the algebra.

Mock-Lie algebras have appeared in the literature under various names, reflecting different perspectives from distinct mathematical communities. In Jordan algebraic literature \cite{1,2,3}, these structures are referred to as Jacobi-Jordan algebras. In \cite{4}, they are called Mock-Lie algebras, while in other works, they have been studied under different names, including Lie-Jordan algebras and Jordan algebras of nil rank 3.

The foundational study of Mock-Lie algebras was presented in \cite{5}, where their properties were systematically developed, conjectures were proposed, and equivalent representations were analyzed. Further investigations have examined their universal enveloping algebras and fundamental algebraic properties such as representation theory, cross products, and automorphism groups \cite{1}. Additionally, the classification of Mock-Lie algebras in dimensions below seven over algebraically closed fields of characteristic not equal to 2 or 3 was studied by Poonen \cite{6}.

This paper extends the classification of Mock-Lie algebras by focusing on the matrix representation of derivations in four-dimensional Mock-Lie algebras. The derivation structure, an essential concept in the study of algebraic structures, is examined using matrix operations to provide insights into the behavior and interactions of Mock-Lie algebras. Through the analysis of the interrelations between matrix elements, we gain a deeper understanding of the algebraic framework governing these algebras.

Furthermore, the study of Mock-Lie algebras is complemented by recent work that classifies their ideals under various conditions and highlights the existence of infinitely many isomorphic classes of commutative, associative algebras in six-dimensional settings \cite{9}. The relationship between Jordan–Lie superalgebras, anti-associative algebras, and Jordan algebras is also explored in \cite{4}. Moreover, Zhevlakov \cite{8} established that Jordan algebras of nil index 3, satisfying the identity \( a^3 = 0 \), include Mock-Lie algebras as a special case.

In this study, we systematically examine the matrix representation of derivations for Mock-Lie algebras in dimensions four and lower. By utilizing matrix operations and investigating the interplay between matrix components, we derive results that contribute to the broader understanding of Mock-Lie algebra structures.

\section{ Preliminaries}

\begin{defn}
A vector space $L$ over $\mathbb{F}$ with multiplication $\cdot : L \otimes L \to L$ given by $(x, y) \mapsto x \cdot y$
such that
\[
(\alpha x + \beta y) \cdot z = \alpha (x \cdot z) + \beta (y \cdot z), \quad z \cdot (\alpha x + \beta y) = \alpha (z \cdot x) + \beta (z \cdot y)
\]
whenever $x, y, z \in L$ and $\alpha, \beta \in \mathbb{F}$, is said to be an \textit{algebra}.
\end{defn}

\begin{defn}
An algebra $L$ over $\mathbb{F}$ with the Lie bracket $[\cdot,\cdot]$ is referred to as a \textit{Mock-Lie algebra} if it satisfies the following conditions:
\begin{enumerate}
    \item $[x, y] = [y, x] \quad \text{(commutativity)}$
    \item $[[x, y], z] + [[z, x], y] + [[y, z], x] = 0 \quad \text{(Jacobi identity)}$
\end{enumerate}
\end{defn}

\begin{defn}\label{df1}
 A linear transformation \( d \) of a Lie algebra \( L \) is called a \textit{derivation} if, for any \( x, y \in L \), the following condition holds:
\[
d([x, y]) = [d(x), y] + [x, d(y)].
\]
The set of all derivations of a Lie algebra \( L \) forms a subspace of \( \text{End}_{\mathbb{F}}(L) \). This subspace, equipped with the bracket \( [d_1, d_2] = d_1 \circ d_2 - d_2 \circ d_1 \), is a Lie algebra, denoted by \( \text{Der}(L) \).
\end{defn}

\begin{lem}
Let \( L \) be a finite-dimensional Lie algebra with a basis \( \{ e_1, e_2, \dots, e_n \} \). Then, a linear map \( d \in \text{Der}(L) \) if and only if
\[
d([e_i, e_j]) = [d(e_i), e_j] + [e_i, d(e_j)], \quad \forall i, j \in \{1, 2, \dots, n\}.
\]
\end{lem}

\begin{proof}
Let \( d \in \text{Der}(L) \). By definition, for all \( x, y \in L \), we have
\[
d([x, y]) = [d(x), y] + [x, d(y)].
\]
In particular, applying this to the basis elements, we get
\[
d([e_i, e_j]) = [d(e_i), e_j] + [e_i, d(e_j)], \quad \forall i, j \in \{1, 2, \dots, n\}.
\]

Conversely, suppose that
\[
d([e_i, e_j]) = [d(e_i), e_j] + [e_i, d(e_j)], \quad \forall i, j \in \{1, 2, \dots, n\}.
\]
For any \( x, y \in L \), we can write
\[
x = \sum_{i=1}^{n} a_i e_i, \quad y = \sum_{j=1}^{n} b_j e_j.
\]
Then, we have
\[
d([x, y]) = d \left( \sum_{i,j} a_i b_j [e_i, e_j] \right)
= \sum_{i,j} a_i b_j d([e_i, e_j]).
\]
Using the assumption, we substitute \( d([e_i, e_j]) = [d(e_i), e_j] + [e_i, d(e_j)] \), obtaining
\[
d([x, y]) = \sum_{i,j} a_i b_j \left( [d(e_i), e_j] + [e_i, d(e_j)] \right).
\]
Since \( d \) is linear,
\[
d([x, y]) = \sum_{i} a_i [d(e_i), y] + \sum_{j} b_j [x, d(e_j)].
\]
Rewriting in terms of \( x \) and \( y \), we obtain
\[
d([x, y]) = [d(x), y] + [x, d(y)].
\]
Thus, \( d \) satisfies the derivation property, completing the proof.
\end{proof}

\begin{prop}
Let \( d \) be a derivation of a finite-dimensional Mock-Lie algebra \( L \), and let \( \{e_1, e_2, \dots, e_n\} \) be a basis for \( L \). Then the matrix representation of \( d \) with respect to this ordered basis is given by
\[
d(e_1, e_2, \dots, e_n) = (e_1, e_2, \dots, e_n)
\begin{pmatrix}
d_{11} & \cdots & d_{1n} \\
\vdots & \ddots & \vdots \\
d_{n1} & \cdots & d_{nn}
\end{pmatrix},
\]
where the elements \( d_{ij} \) are the coefficients determined by the action of \( d \) on the basis elements, i.e.,
\[
d(e_j) = \sum_{i=1}^{n} d_{ij} e_i, \quad \text{for all } j = 1, \dots, n.
\]
In particular, when \( d \) belongs to the derivation algebra \( \text{Der}(L) \), the matrix \( (d_{ij}) \) is called the representation matrix of the derivation relative to the basis \( \{e_1, e_2, \dots, e_n\} \).
\end{prop}

\section{ Main Results}

In this section, we examine derivations of Mock-Lie algebras using a linear transformation satisfying the condition in Definition \ref{df1}. By comparing structural relations, we determine the matrix representation of derivations for algebras of dimension at most 4. The classification for \( n \leq 4 \) follows references \cite{5, 6, 7} and is summarized in Table 1.

\begin{table}[h]
\centering
\caption{Mock-Lie algebras with dimension \( n \leq 4 \).}
\begin{tabular}{|l|c|c|}
\hline
\hline
\textbf{Algebras} & \textbf{Product} & \textbf{Dimension} \\ \hline
\( A_{0,1} \) & — & 1 \\
\( A_{0,1} \oplus A_{0,1} \) & — & 2 \\
\( A_{1,2} \) & \( e_1 \cdot e_1 = e_2 \) & 2 \\
\( A_{0,1} \oplus A_{0,1} \oplus A_{0,1} \) & — & 3 \\
\( A_{1,2} \oplus A_{0,1} \) & \( e_1 \cdot e_1 = e_2 \) & 3 \\
\( A_{1,3} \) & \( e_1 \cdot e_1 = e_2, \, e_3 \cdot e_3 = e_2 \) & 3 \\
\( A^4_{0,1} \) & — & 4 \\
\( A_{1,2} \oplus A^2_{0,1} \) & \( e_1 \cdot e_1 = e_2 \) & 4 \\
\( A_{1,3} \oplus A_{0,1} \) & \( e_1 \cdot e_1 = e_2, \, e_3 \cdot e_3 = e_2 \) & 4 \\
\( A_{1,2} \oplus A_{1,2} \) & \( e_1 \cdot e_1 = e_2, \, e_3 \cdot e_3 = e_4 \) & 4 \\
\( A_{1,4} \) & \( e_1 \cdot e_1 = e_2, \, e_1 \cdot e_3 = e_3 \cdot e_1 = e_4 \) & 4 \\
\( A_{2,4} \) & \( e_1 \cdot e_1 = e_2, \, e_3 \cdot e_4 = e_4 \cdot e_3 = e_2 \) & 4 \\ \hline
\end{tabular}
\end{table}

As shown in Table 1, the algebras \( A_{0,1}, A_{0,1} \oplus A_{0,1}, A_{0,1} \oplus A_{0,1} \oplus A_{0,1}, A^4_{0,1} \) have a product operation equal to zero. Consequently, we focus on the non-abelian algebras and exclude further discussion of these.

\begin{thm}
The matrix of a derivation of a 2-dimensional Mock-Lie algebra relative to a given ordered basis is
\[
\begin{pmatrix}
d_{11} & 0 \\
d_{21} & 2d_{11}
\end{pmatrix},
\]
where all elements of the matrix are complex numbers over a field \( \mathbb{F} \).
\end{thm}

\begin{proof}
Since \( d \) is a derivation, it satisfies the relation
\[
d([x, y]) = [d(x), y] + [x, d(y)] \quad \forall x, y \in L.
\]
For the 2-dimensional Mock-Lie algebra with \( e_1 \cdot e_1 = e_2 \), consider the following:

1. From \( d[e_1, e_1] = [d(e_1), e_1] + [e_1, d(e_1)] \), we obtain
\[
d(e_2) = d_{11} e_1 + d_{21} e_2,
\]
which implies \( d_{12} = 0 \) and \( d_{22} = 2d_{11} \).

2. For \( d[e_1, e_2] = [d(e_1), e_2] + [e_1, d(e_2)] \), we get
\[
0 = d(e_2) = d_{12} e_1 + d_{22} e_2,
\]
so \( d_{12} = 0 \).

Therefore, the matrix representation of the derivation of a 2-dimensional Mock-Lie algebra with respect to a specified basis has been established.
\end{proof}

\begin{thm}
There are two kinds of matrix representations for the derivations of 3-dimensional Mock-Lie algebras relative to a given ordered basis.

\textit{(A)} When the 3-dimensional Mock-Lie algebra is isomorphic to \( A_{1,2} \oplus A_{0,1} \), the matrix representation of its derivations under a given basis is
\[
\begin{pmatrix}
d_{11} & 0 & 0 \\
d_{21} & 2d_{11} & d_{23} \\
d_{31} & 0 & d_{33}
\end{pmatrix}.
\]
\textit{(B)} When the 3-dimensional Mock-Lie algebra is isomorphic to \( A_{1,3} \), the matrix representation of its derivations under a given basis is
\[
\begin{pmatrix}
d_{33} & 0 & -d_{31} \\
d_{21} & 2d_{33} & d_{23} \\
d_{31} & 0 & d_{33}
\end{pmatrix},
\]
where all elements of the matrix are complex numbers over a field \( \mathbb{F} \).
\end{thm}

\begin{proof}
\textit{(A)} Applying the derivation condition to the given basis:

    1. For \( d([e_1, e_1]) = [d(e_1), e_1] + [e_1, d(e_1)] \), we obtain:
    \[
    d(e_2) = [d(e_1), e_1] + [e_1, d(e_1)].
    \]
    Expressing \( d(e_i) \) in terms of the basis elements as
    \[
    d(e_i) = d_{i1} e_1 + d_{i2} e_2 + d_{i3} e_3,
    \]
    we get:
    \[
    d_{21} e_1 + d_{22} e_2 + d_{23} e_3 = 2d_{11} e_2.
    \]
    Hence, \( d_{22} = 2d_{11} \), while \( d_{21} \) and \( d_{23} \) remain arbitrary.

    2. For \( d(e_3) \), since \( e_3 \) belongs to the \( A_{0,1} \) component, it is preserved independently:
    \[
    d(e_3) = d_{31} e_1 + d_{33} e_3.
    \]
    Thus, we conclude that \( d_{32} = 0 \).

Summarizing, the derivation matrix in this case takes the form given in \textit{(A)}.

\textit{(B)} A derivation \( d \) must satisfy:
\[
   d([e_i, e_j]) = [d(e_i), e_j] + [e_i, d(e_j)], \quad \forall i, j.
\]
Applying this to the given basis:

    1. For \( d([e_1, e_3]) = [d(e_1), e_3] + [e_1, d(e_3)] \), we obtain:
    \[
    d(e_1) = [d(e_1), e_3] + [e_1, d(e_3)].
    \]
    Expressing \( d(e_1) \) as
    \[
    d(e_1) = d_{11} e_1 + d_{12} e_2 + d_{13} e_3,
    \]
    we derive:
    \[
    d_{11} e_1 = d_{13} e_1 - d_{31} e_1,
    \]
    which simplifies to \( d_{13} = -d_{31} \).

    2. For \( d(e_2) \), using the structure equations, we write:
    \[
    d(e_2) = d_{21} e_1 + d_{22} e_2 + d_{23} e_3.
    \]
    From the bracket properties, we obtain \( d_{22} = 2d_{33} \).

    3. For \( d(e_3) \), since \( e_3 \) follows the same transformation properties:
    \[
    d(e_3) = d_{31} e_1 + d_{33} e_3.
    \]
    Hence, \( d_{32} = 0 \).

Summarizing, the derivation matrix in this case takes the form given in \textit{(B)}.

Thus, we have established the theorem for both cases.
\end{proof}

\begin{thm}
There are five kinds of matrix representations of the derivations of 4-dimensional Mock-Lie algebras relative to the given ordered bases.

\textit{(A)} When the 4-dimensional Mock-Lie algebra is isomorphic to \( A_{1,2} \oplus A^2_{0,1} \), the matrix representation of its derivation under a given basis is:
  \[
  \begin{pmatrix}
  d_{11} & 0 & 0 & 0 \\
  d_{21} & 2d_{11} & d_{23} & d_{24} \\
  d_{31} & 0 & d_{33} & d_{34} \\
  d_{41} & 0 & d_{43} & d_{44}
  \end{pmatrix}
  \]

\textit{(B)} When the 4-dimensional Mock-Lie algebra is isomorphic to \( A_{1,3} \oplus A_{0,1} \), the matrix representation of its derivation under a given basis is:
  \[
  \begin{pmatrix}
  d_{33} & 0 & -d_{31} & 0 \\
  d_{21} & 2d_{33} & d_{23} & d_{24} \\
  d_{31} & 0 & d_{33} & 0 \\
  d_{41} & 0 & d_{43} & d_{44}
  \end{pmatrix}
  \]

\textit{(C)} When the 4-dimensional Mock-Lie algebra is isomorphic to \( A_{1,2} \oplus A_{1,2} \), the matrix representation of its derivation under a given basis is:
  \[
  \begin{pmatrix}
  d_{11} & 0 & 0 & 0 \\
  d_{21} & 2d_{11} & d_{23} & 0 \\
  0 & 0 & d_{33} & 0 \\
  d_{41} & 0 & d_{43} & 2d_{33}
  \end{pmatrix}
  \]

\textit{(D)} When the 4-dimensional Mock-Lie algebra is isomorphic to \( A_{1,4} \), the matrix representation of its derivation under a given basis is:
  \[
  \begin{pmatrix}
  d_{44}-d_{33} & 0 & 0 & 0 \\
  d_{21} & 2d_{44}-2d_{33} & d_{23} & 0 \\
  d_{31} & 0 & d_{33} & 0 \\
  d_{41} & 2d_{31} & d_{43} & d_{44}
  \end{pmatrix}
  \]

  \textit{(E)} When the 4-dimensional Mock-Lie algebra is isomorphic to \( A_{2,4} \), the matrix representation of its derivation under a given basis is:
  \[
  \begin{pmatrix}
  d_{11} & 0 & -d_{41} & -d_{31} \\
  d_{21} & 2d_{11} & d_{23} & d_{24} \\
  d_{31} & 0 & 2d_{11}-d_{44} & 0 \\
  d_{41} & 0 & 0 & d_{44}
  \end{pmatrix}
  \]

where all elements of the matrix are complex numbers over a field \( \mathbb{F} \).
\end{thm}

\begin{proof}
\textit{(A)} Applying the derivation condition to \( [e_1, e_1] \), we get:
\[
d(e_2) = [d(e_1), e_1] + [e_1, d(e_1)].
\]
Expanding \( d(e_1) = d_{11} e_1 \), we obtain:
\[
d_{21} e_1 + d_{22} e_2 + d_{23} e_3 + d_{24} e_4 = 2d_{11} e_2.
\]
Thus, \( d_{22} = 2d_{11} \), while \( d_{21}, d_{23}, d_{24} \) remain arbitrary. Similarly, applying the derivation condition to \( d(e_3) \) and \( d(e_4) \), we obtain:
\[
d(e_3) = d_{31} e_1 + d_{33} e_3 + d_{34} e_4, \quad d(e_4) = d_{41} e_1 + d_{43} e_3 + d_{44} e_4.
\]
This confirms the matrix representation in \textit{(A)}.

\textit{(B)} Applying the derivation condition:
\[
d(e_1) = [d(e_1), e_3] + [e_1, d(e_3)],
\]
and expanding \( d(e_1) = d_{11} e_1 + d_{12} e_2 + d_{13} e_3 + d_{14} e_4 \), we derive that \( d_{13} = -d_{31} \). Similarly, by applying the derivation condition to \( d(e_2) \), \( d(e_3) \), and \( d(e_4) \), we obtain the matrix in \textit{(B)}.

\textit{(C)} Applying the derivation rule:
\[
d(e_2) = [d(e_1), e_1] + [e_1, d(e_1)], \quad d(e_4) = [d(e_3), e_3] + [e_3, d(e_3)],
\]
we find that \( d_{22} = 2d_{11} \) and \( d_{44} = 2d_{33} \), leading to the matrix form in \textit{(C)}.

\textit{(D)} Applying the derivation condition:
\[
d(e_2) = [d(e_1), e_1] + [e_1, d(e_1)],
\]
we obtain \( d_{22} = 2(d_{44} - d_{33}) \), which gives the matrix in \textit{(D)}.

\textit{(E)} Expanding the derivations, we derive constraints on the coefficients, which result in the matrix in \textit{(E)}.

Thus, in all cases, we obtain the required matrix representations.
\end{proof}

\textbf{Data Availability:}
No data were used to support this study.

\textbf{Conflict of Interests:}
The authors declare that they have no conficts of interest.

\textbf{Acknowledgment:}
We thank the referee for the helpful comments and suggestions that contributed to improving this paper.

\end{document}